\documentclass{amsart}
\usepackage{graphicx, amssymb, eucal, hyperref}
\newtheorem{iThm}{Theorem}

\newtheorem{thm}{Theorem}[section]
\newtheorem{cor}[thm]{Corollary}
\newtheorem{lem}[thm]{Lemma}
\newtheorem{prop}[thm]{Proposition}
\theoremstyle{definition}
\newtheorem{defn}[thm]{Definition}
\theoremstyle{remark}
\newtheorem{rem}[thm]{Remark}
\newtheorem{claim}{Claim}[thm]
\numberwithin{equation}{section}
\newcommand{\vertiii}[1]{{\left\vert\kern-0.25ex\left\vert\kern-0.25ex\left\vert #1 
    \right\vert\kern-0.25ex\right\vert\kern-0.25ex\right\vert}}
\newcommand{\norm}[1]{\left\Vert#1\right\Vert}
\newcommand{\abs}[1]{\left\vert#1\right\vert}
\newcommand{\set}[1]{\left\{\,#1\,\right\}}
\newcommand{\defined}[1]{\emph{#1}}
\newcommand{\mc}[1]{\mathcal{#1}}
\newcommand{\op}[1]{\operatorname{#1}}
\newcommand{\ignore}[1]{}
\newcommand{\LI}{\to}
\newcommand{\IL}{\mc{L}_{\omega_1, \omega}}
\newcommand{\mb}[1]{\mathbf{#1}}
\newcommand{\sqeq}{\trianglelefteq}
\newcommand{\sq}{\triangleleft}
\newcommand{\Disc}{\op{Discrete}}
\newcommand{\tup}[1]{\left\langle\,#1\,\right\rangle}
\newcommand{\simneg}{\raise.17ex\hbox{$\scriptstyle\sim$}}
\renewcommand{\phi}{\varphi}
\begin{document}

\title[Omitting types]{Omitting types for infinitary $[0, 1]$-valued logic}%
\author{Christopher J. Eagle}%
\thanks{Research partially supported by an NSERC PGS-D award}
\address{University of Toronto, Department of Mathematics, 40 St. George St., Toronto, Ontario, Canada M5S 2E4}%
\email{cjeagle@math.toronto.edu}%


\date{\today}%
% ----------------------------------------------------------------
\begin{abstract}
We describe an infinitary logic for metric structures which is analogous to $\IL$.  We show that this logic is capable of expressing several concepts from analysis that cannot be expressed in finitary continuous logic.  Using topological methods, we prove an omitting types theorem for countable fragments of our infinitary logic.  We use omitting types to prove a two-cardinal theorem, which yields a strengthening of a result of Ben Yaacov and Iovino concerning separable quotients of Banach spaces.
\end{abstract}
\maketitle
% ----------------------------------------------------------------
\section*{Introduction}
\renewcommand\theiThm{\Alph{iThm}}
In this paper we study an infinitary logic for structures based on metric spaces.  Our logic is the natural generalization of $\IL$ to the setting of metric structures, and our main result is an Omitting Types Theorem for this logic.

Real-valued logics have had a variety of applications in analysis, beginning with the introduction of ultrapowers of Banach spaces by Dacunha-Castelle and Krivine in \cite{Dacunha-Castelle1970}.  Krivine, and later Stern, used this approach to solve important problems in functional analysis \cite{Krivine1976}, \cite{Krivine1981}, \cite{Stern1978}.  See also \cite{Krivine1972}, \cite{Krivine1974}, \cite{Krivine1984}.  In \cite{Chang1966} Chang and Keisler develop a general framework for continuous model theory with truth values in a fixed compact Hausdorff space $K$, with the case $K = [0, 1]$ being the motivating example.  In recent years there has been a considerable amount of activity in the $[0, 1]$-valued logic known as first-order continuous logic, which Ben Yaacov and Usvyatsov introduced in \cite{BenYaacov2010} as a reformulation of Henson's logic for Banach spaces (see \cite{Henson2003}) in the framework of \cite{Chang1966}.  See \cite{BenYaacov2008a} for a self-contained introduction to first-order continuous logic.

In this paper we extend the $[0, 1]$-valued logics mentioned above by allowing formulas with conjunctions and disjunctions of countable length, which makes our logic a $[0, 1]$-valued version of $\IL$.  Including infinitary formulas in our logic allows us to axiomatize important classes of structures from functional analysis (see Section~\ref{SectionInfinitary}).  Our logic does not satisfy the compactness theorem, but as our main result shows, countable fragments of this logic do satisfy the classical Omitting Types Theorem.

In the $[0, 1]$-valued setting we say that a type is \defined{principal} over a theory $T$ if there is a formula $\phi$ consistent with $T$, and an approximation $\phi'$ of $\phi$, such that in models of $T$ elements satisfying $\phi'$ are realizations of $\Sigma$ (see Definition~\ref{defnPrincipal2}).  Our main result is the following:

\begin{iThm}\label{thmOmittingTypesStatement}
Let $S$ be a metric signature, and let $L$ be a countable fragment of $\IL(S)$.  Let $T$ be an $L$-theory.  For each $n < \omega$, let $\Sigma_n$ be a type of $T$ that is not principal over $T$.  Then there is a model of $T$ that omits every $\Sigma_n$.
\end{iThm}

Each type $\Sigma_n$ is in finitely many variables, but the number of variables may increase with $n$.  For certain fragments of $\IL$, which we call \defined{continuous fragments}, Theorem~\ref{thmOmittingTypesStatement} implies an Omitting Types Theorem in which the resulting model is based on a complete metric space (Proposition~\ref{propMetricOmitting}).  This latter version generalizes Henson's Omitting Types Theorem to infinitary languages (see \cite{BenYaacov2007}).  Henson's theorem is the motivation for recent uses of omitting types to characterize certain classes of operator algebras \cite{Carlson}.

In the finitary setting it is straightforward to generalize the Omitting Types Theorem to uncountable languages (see Theorem \cite[Theorem 2.2.19]{Chang1990}).  However, the main result of \cite{Caicedo2012} shows that Theorem~\ref{thmOmittingTypesStatement} cannot be generalized to arbitrary uncountable fragments of $\IL$.

Our approach is topological, based on the connection between omitting types and the Baire Category Theorem.  The Omitting Types Theorem for the classical (discrete) $\IL$ is originally due to Keisler in \cite{Keisler1971}.  Later Morley \cite{Morley1974} obtained the same result by showing that a relevant topological space is metrizable by a complete metric.  In our proof of Theorem~\ref{thmOmittingTypesStatement} we avoid the issue of metrizability by instead working with a topological notion of completeness.  Applying our proof in the classical setting thus gives a simplification of Morley's argument.  We have learned from Iovino \cite{IovinoPrivate} that Caicedo has independently, in unpublished work, obtained the conclusion of Theorem~\ref{thmOmittingTypesStatement} by adapting Morley's argument.

As applications of Theorem~\ref{thmOmittingTypesStatement} we prove a $[0, 1]$-valued version of Keisler's two-cardinal theorem (Theorem~\ref{thmKeisler}), and we extend a result of Ben Yaacov and Iovino from \cite{benYaacov2009} regarding non-trivial separable quotients of Banach spaces (Corollary~\ref{corBanach}).

The paper requires only basic knowledge of classical first-order model theory.  In Section~\ref{SectionPreliminaries} we provide the needed background material from both general topology and the model theory of metric structures.  In Section~\ref{SectionInfinitary} we introduce our $[0, 1]$-valued version of $\IL$ and state some of its basic properties.  The topological spaces we will use in the proof of Theorem~\ref{thmOmittingTypesStatement} are described in Section~\ref{SectionTopology}.  In Section~\ref{SectionOmittingTypes} we prove that the topological space from the previous section satisfy a topological completeness property which implies that they are Baire spaces.  We then show that these topological results imply Theorem~\ref{thmOmittingTypesStatement}.  Section~\ref{SectionApplication} contains the two-cardinal theorem and the applications to separable quotients of Banach spaces.

\textbf{Acknowledgements: } We would like to thank Frank Tall for suggesting the use of \v{C}ech-completeness as a replacement for compactness or metrizability in the study of the spaces of Section~\ref{SectionTopology}.  We also wish to thank Jos\'e Iovino for helpful comments on earlier versions of this paper.  Finally, we thank the anonymous referee for their helpful comments.  Most of Section~\ref{subsectionProof}, as well as part of Section~\ref{subsectionCompleteStructures}, follows \cite{Caicedo2012} closely, and is included here for completeness.

\section{Preliminaries}\label{SectionPreliminaries}
\subsection{Topological preliminaries}
The spaces we will use are not $T_0$, but do have other separation properties.  In particular, we say that a space is completely regular to mean that points and closed sets can be separated by continuous functions, without assuming the Hausdorff condition.  The most important notion from topology for us is the notion of Baire category.  Recall that if $X$ is a topological space and $A \subseteq X$, then $A$ is \defined{nowhere dense} if $\op{int}(\overline{A}) = \emptyset$.  A space $X$ is \defined{Baire} if whenever $A_n, n < \omega$ are closed nowhere dense subsets of $X$, then $X \setminus \left(\bigcup_{n < \omega}A_n\right)$ is dense in $X$.  The classical Baire Category Theorem states that locally compact Hausdorff spaces and completely metrizable spaces are Baire.  The spaces we will be considering are neither locally compact nor metrizable, but they do have a more general property, which we shall now describe.  Recall that a family $\mc{F}$ of sets is \defined{centred} if $\bigcap \mc{F}' \neq \emptyset$ for every finite $\mc{F}' \subseteq \mc{F}$.

\begin{defn}\label{defnCechComplete}
Let $X$ be a completely regular space.  A \defined{complete sequence of open covers of $X$} is a sequence $\tup{\mc{U}_n : n < \omega}$ of open covers of $X$ with the following property: If $\mc{F}$ is a centred family of closed subsets of $X$ such that for each $n < \omega$ there is $F_n \in \mc{F}$ and $U_n \in \mc{U}_n$ such that $F_n \subseteq U_n$, then $\bigcap \mc{F} \neq \emptyset$.

A completely regular space $X$ is \defined{\v{C}ech-complete} if there exists a complete sequence of open covers of $X$.
\end{defn} 

If the space $X$ is completely regular and Hausdorff then $X$ is \v{C}ech-complete if and only if $X$ is a $G_\delta$ subspace of some (equivalently, every) compactification.  For metrizable spaces, being \v{C}ech-complete is equivalent to being completely metrizable.  It follows from these two facts that if $X$ is either locally compact Hausdorff or completely metrizable then $X$ is \v{C}ech-complete.  The following result states the two key facts about \v{C}ech-complete spaces that we will use in the proof of Theorem \ref{thmOmittingTypesStatement}.  These facts are stated and proved in \cite{Engelking1989} for completely regular Hausdorff spaces, but the proof does not use the Hausdorff condition.

\begin{lem}\label{lemCompleteProperties}
Let $X$ be a completely regular space.
\begin{enumerate}
\item{
If $X$ is \v{C}ech-complete then $X$ is Baire.
}
\item{
If $X$ is \v{C}ech-complete and $F \subseteq X$ is a closed subspace, then $F$ is \v{C}ech-complete.
}
\end{enumerate}
\end{lem}

\begin{rem}\label{remSetTheoretic}
Under additional set-theoretic assumptions, the first part of Lemma~\ref{lemCompleteProperties} can be improved, as follows.  For an infinite cardinal $\kappa$, a space $X$ is \defined{$\kappa$-Baire} if the intersection of fewer than $\kappa$ dense open subsets of $X$ is dense in $X$.  In this terminology our previous definition of Baire corresponds to $\aleph_1$-Baire.  Recall that a space $X$ has the \defined{countable chain condition} if every family of pairwise disjoint open subsets of $X$ is at most countable.  Tall \cite[Theorem 2.3]{Tall1974} observed that Martin's Axiom implies that \v{C}ech-complete spaces with the countable chain condition are $2^{\aleph_0}$-Baire.  Essentially the same proof shows that Martin's Axiom restricted to countable partial orders implies that any \v{C}ech-complete space with a countable base is $2^{\aleph_0}$-Baire.  See \cite{Fremlin1984} for details about Martin's Axiom.

The proof in \cite{Tall1974} assumes the Hausdorff condition, but the result for \v{C}ech-complete Hausdorff spaces implies the same result for arbitrary \v{C}ech-complete spaces.  If $X$ is a \v{C}ech-complete space and $\equiv$ is the relation of topological indistinguishability, then $X / \equiv$ is a \v{C}ech-complete Hausdorff space.  It is routine to check that for any cardinal $\kappa$, if $X / \equiv$ is $\kappa$-Baire then so is $X$.
\end{rem}

\subsection{Metric structures}
Our results in this paper concern model theory for structures based on metric spaces.  In order to keep the paper self-contained, we present here the necessary background material about these structures.  If $(M, d)$ is a metric space and $n \in \omega$, then we consider $M^n$ as a metric space in the metric $d'$ defined by $d'((x_1, \ldots, x_n), (y_1, \ldots, y_n)) = \sup_{1 \leq i \leq n}d(x_i, y_i)$.

\begin{defn}
A \defined{metric structure} $\mc{M}$ consists of the following:
\begin{itemize}
\item{
A metric space $(M, d)$, where the metric $d$ is bounded by $1$,
}
\item{
A set of uniformly continuous functions of the form $f : M^n \to M$,
}
\item{
A set of uniformly continuous functions of the form $P : M^n \to [0, 1]$,
}
\item{
A set of distinguished elements of $M$.
}
\end{itemize}
\end{defn}

We say that the structure $\mc{M}$ is based on the metric space $(M, d)$.  We allow the possibility that the sets of distinguished functions and elements may be empty.  Observe that if $\mc{M}$ is a metric structure where the metric $d$ is the discrete metric, and each $P : M^n \to [0, 1]$ actually takes values in $\{0, 1\}$, then $\mc{M}$ is a structure in the usual sense of first-order logic.  We call such structures \defined{discrete structures}.  For certain structures, such as those based on normed spaces, the requirement that metric structures have metrics bounded by $1$ can be overcome by taking the structure of the unit ball of the normed space, rather than the entire space.  Alternatively, one can introduce a multi-sorted version of metric structures, and then take as sorts scaled versions of the closed balls of integer radius.  

In many examples it is desirable to have metric structures where the underlying metric spaces are complete.  Suppose that $\mc{M}$ is a metric structure based on the metric space $(M, d)$.  Then since all of the distinguished functions are uniformly continuous, they extend uniquely to the completion $(\overline{M}, \overline{d})$ of $(M, d)$.  We denote by $\overline{\mc{M}}$ the structure based on $(\overline{M}, \overline{d})$, with the distinguished functions given by the unique extensions of the distinguished functions of $\mc{M}$, and call this structure the \defined{completion} of $\mc{M}$.

Various special classes of metric structures have been considered in the literature.  For example, metric structures with $1$-Lipschitz functions and predicates are the structures used in \L ukasiewicz-Pavelka logic \cite{Hajek1998}.  Metric structures based on complete metric spaces are the subject of continuous logic \cite{BenYaacov2010}.

The metric structures form the semantic objects for $[0, 1]$-valued model theory.  We now introduce the corresponding syntax.  By a \defined{modulus of continuity} for a uniformly continuous function $f : M^n \to M$ we mean a function $\delta : \mathbb{Q} \cap (0, 1) \to \mathbb{Q} \cap (0, 1)$ such that such that for all $a_1, \ldots, a_n, b_1, \ldots, b_n \in M$ and all $\epsilon \in \mathbb{Q} \cap (0, 1)$,
\[\sup_{1 \leq i \leq n}d(a_i, b_i) < \delta(\epsilon) \implies  d(f(a_i), f(b_i)) \leq \epsilon.\]
Similarly, $\delta$ is a modulus of continuity for $P : M^n \to [0, 1]$ means that for all $a_1, \ldots, a_n, b_1, \ldots, b_n \in M$,
\[\sup_{1 \leq i \leq n}d(a_i, b_i) < \delta(\epsilon) \implies  \abs{P(a_i) - P(b_i)} \leq \epsilon.\]

\begin{defn}
A \defined{metric signature} consists of the following sets, any of which may be empty:
\begin{itemize}
\item{
A set of function symbols, each with an associated arity and modulus of continuity,
}
\item{
A set of predicate symbols, each with an associated arity and modulus of continuity,
}
\item{
A set of constant symbols.
}
\end{itemize}
\end{defn}

If $S$ is a metric signature and $\mc{M}$ is a metric structure, then $\mc{M}$ is a \defined{$S$-structure} if the distinguished functions of $\mc{M}$ have the moduli of continuity of the corresponding symbols of $S$.  If $\mc{M}$ and $\mc{N}$ are $S$-structures we have the natural notion of \defined{substructure}, which we denote by $\mc{M} \subseteq \mc{N}$.

We will need the notion of an ultraproduct of metric structures.  Ultraproducts of metric spaces were introduced independently by Krivine in the context of Banach space theory \cite{Krivine1967}, and by Luxemburg in the context of nonstandard hulls of metric spaces \cite{Luxemburg1969}.  We will only sketch the construction here -- for details, see \cite[\textsection 5]{BenYaacov2008a}.  If $X$ is a topological space, $(x_\alpha)_{\alpha < \kappa}$ is a $\kappa$-sequence of points from $X$, and $\mc{D}$ is an ultrafilter on $\kappa$, then we denote the ultrafilter limit of the $x_\alpha$'s along $\mc{D}$ (when it exists) by $\lim_{\alpha \to \mc{D}}x_\alpha$.  Suppose that $\kappa$ is an infinite cardinal, and that $\tup{M_\alpha, d_\alpha}$ is a metric space of diameter at most $1$ for each $\alpha < \kappa$.  Let $\mc{D}$ be an ultrafilter on $\kappa$.  Define a function $d$ on the cartesian product $\prod_{\alpha < \kappa}M_\alpha$ by
\[d((x_\alpha)_{\alpha < \kappa}, (y_\alpha)_{\alpha < \kappa}) = \lim_{\alpha \to \mc{D}}d_\alpha(x_\alpha, y_\alpha).\]
The function $d$ is a pseudometric on $\prod_{\alpha < \kappa}M_\alpha$.  We define the \defined{ultraproduct} of the $M_\alpha$'s to be the metric space obtained from the pseudometric space $\tup{\prod_{\alpha < \kappa}M_\alpha, d}$ by taking the quotient by the relation $d(x, y) = 0$.  We denote the ultraproduct by $\prod_{\mc{D}}M_\alpha$.

Now suppose that for each $\alpha < \kappa$ we have a uniformly continuous function $f_\alpha : M_\alpha^n \to M_\alpha$, and that there is a single function $\delta$ that is a modulus of uniform continuity for each $f_\alpha$.  For each $1 \leq i \leq n$, write $[x^i]_{\mc{D}}$ for the equivalence class of $(x^i_\alpha)_{\alpha < \kappa}$ in $\prod_{\mc{D}}M_\alpha$.  We define the \defined{ultraproduct} of the $f_\alpha$'s to be the function $f : \left(\prod_{\mc{D}}M_\alpha\right)^n \to \prod_{\mc{D}}M_\alpha$ defined as follows:
\[f([x^1]_{\mc{D}}, \ldots, [x^n]_{\mc{D}}) = [f_\alpha(x^1_\alpha, \ldots, x^n_\alpha)]_{\mc{D}}.\]
It follows from the assumption that $\delta$ is a modulus of uniform continuity for every $f_\alpha$ that $f$ is well-defined and also has $\delta$ as a modulus of uniform continuity.  Using the fact that $\prod_{\mc{D}}[0, 1] = [0, 1]$ for any ultrafilter $\mc{D}$, we can similarly define the ultraproduct of functions $P_\alpha : M_\alpha^n \to [0, 1]$.  Finally, if $a_\alpha \in M_\alpha$ for each $\alpha$, define the ultraproduct of the $a_\alpha$'s to be $[a_\alpha]_{\mc{D}} \in \prod_{\mc{D}}M_\alpha$.  We can now define the ultraproduct of metric structures:

\begin{defn}\label{defnUltraproduct}
Fix a metric signature $S$ and an infinite cardinal $\kappa$.  For each $\alpha < \kappa$, let $\mc{M}_\alpha = \tup{M_\alpha, d_\alpha, \ldots}$ be an $S$-structure.  Then the \defined{ultraproduct} $\prod_{\mc{D}}\mc{M}_\alpha$ is the $S$-structure whose underlying metric space is $\prod_{\mc{D}}M_\alpha$, with each symbol in $S$ interpreted as the ultraproduct of the interpretations in the $\mc{M}_\alpha$'s.
\end{defn}

\section{$\IL$ for metric structures}\label{SectionInfinitary}
In this section we describe a natural analogue of $\mc{L}_{\omega_1, \omega}$ in the $[0, 1]$-valued setting.  Given a metric signature $S$, the \defined{terms} of $S$ are defined exactly as in classical first-order logic.

\begin{defn}
Let $S$ be a metric signature.  An \defined{atomic formula} of $S$ is an expression of one of the following two forms:
\begin{itemize}
\item{
$d(t_1(\overline{x}), t_2(\overline{x}))$, where $t_1$ and $t_2$ are terms,
}
\item{
$P(t_1(\overline{x}), \ldots, t_n(\overline{x}))$, where $t_1, \ldots, t_n$ are terms.
}
\end{itemize}
\end{defn}

\begin{defn}
Let $S$ be a metric signature.  We define the formulas of $\IL(S)$ recursively, as follows:
\begin{enumerate}
\item{
Each atomic formula is a formula of $\IL(S)$.
}
\item{
If $\phi, \psi$ are formulas of $\IL(S)$ then so is $\phi \LI \psi$.
}
\item{
For each $q \in \mathbb{Q} \cap [0, 1]$, $q$ is a formula of $\IL(S)$.
}
\item{
If $\phi_0, \phi_1, \ldots$ are formulas of $\IL(S)$, and the total number of free variables used in all the $\phi_n$'s is finite, then $\sup_{n < \omega}\phi_n$ is a formula.
}
\item{
If $\phi$ is a formula of $\IL(S)$ then so is $\sup_x \phi$.
}
\end{enumerate}
\end{defn}

The semantics for this logic is given by a recursive definition analogous to the definition of satisfaction in the classical setting.  For a metric signature $S$, an $S$-term $t(\overline{x})$, and an $S$-structure $\mc{M}$, we denote by $t^{\mc{M}}$ the function $t^{\mc{M}} : M^n \to M$ determined by $t$.

\begin{defn}
Let $S$ be a metric signature, and $\mc{M}$ an $S$-structure with $\overline{a}$ a tuple from $\mc{M}$.  For $\phi(\overline{x})$ an $S$-formula, we define the \defined{truth value} $\phi^{\mc{M}}(\overline{a})$ as follows:
\begin{itemize}
\item{
If $\phi$ is $d(t_1(\overline{x}), t_2(\overline{x})$ then $\phi^{\mc{M}}(\overline{a}) = d^{\mc{M}}(t_1^{\mc{M}}(\overline{a}), t_2^{\mc{M}}(\overline{a}))$, and similarly if $\phi$ is $P(t_1(\overline{x}), \ldots, t_m(\overline{x}))$.
}
\item{
If $\phi$ is $q$, where $q \in \mathbb{Q} \cap [0, 1]$, then $\phi^{\mc{M}} = q$.
}
\item{
If $\phi$ is $\psi \LI \theta$, then $\phi^{\mc{M}}(\overline{a}) = \min\set{1-\psi^{\mc{M}}(\overline{a}) + \theta^{\mc{M}}(\overline{a}), 1}$.
}
\item{
If $\phi$ is $\sup_{n < \omega}\phi_n$, then $\phi^{\mc{M}}(\overline{a}) = \sup_{n < \omega}\phi^{\mc{M}}(\overline{a})$.
}
\item{
If $\phi$ is $\sup_y\psi(y, \overline{x})$, then $\psi^{\mc{M}}(\overline{a}) = \sup_{y \in M}\psi^{\mc{M}}(y, \overline{a})$.
}
\end{itemize}
\end{defn}

Note that we always have $\phi^{\mc{M}}(a_1, \ldots, a_n) \in [0, 1]$.  We write $\mc{M} \models \phi(a_1, \ldots, a_n)$ to mean $\phi^{\mc{M}}(a_1, \ldots, a_n) = 1$.  Note that we do not have negation in the classical sense; for an arbitrary sentence $\sigma$, there may be no sentence $\simneg\sigma$ such that for each $\mc{M}$ we have $\mc{M} \models \simneg\sigma$ if and only if $\mc{M} \not\models \sigma$.  In particular, we can express ``$x=y$" by the formula ``$1-d(x,y)$", but we cannot in general express ``$x\neq y$". 

The interpretation of the connective $\LI$ will be particularly important for us later, because of the following observation:
\[\mc{M} \models \phi(\overline{a}) \to \psi(\overline{a}) \iff \phi^{\mc{M}}(\overline{a}) \leq \psi^{\mc{M}}(\overline{a}).\]

The connective $\LI$ allows us to define the connectives commonly used in first-order continuous logic.  We define $\neg$, $\wedge$, and $\vee$ as follows:
\[\neg\phi = \phi \LI 0, \qquad \phi \vee \psi = (\phi \to \psi) \to \psi, \qquad \phi \wedge \psi = \neg(\neg \phi \vee \neg \psi).\]
Then it is easy to check that $\mc{M} \models \phi \wedge \psi$ if and only if $\mc{M} \models \phi$ and $\mc{M} \models \psi$.  Similarly, $\vee$ corresponds to disjunction, and $(\neg\phi)^{\mc{M}} = 1 - \phi^{\mc{M}}$.  We also have the truncated addition, defined as $\phi \dotplus \psi = \phi \LI \neg\psi$.  

\begin{rem}
Using only the \L ukasiewicz implication $\LI$ and the rational constants we can obtain the function $x \mapsto \frac{x}{2}$ as a limit (see \cite[Proposition 1.18]{Caicedo2012}):
\[\frac{1}{2}x = \lim_{n\to\infty}\bigvee_{i=1}^{n}\left(\frac{i}{n} \wedge \neg(x \LI \frac{i}{n})\right).\]
Once we have the connective $\frac{x}{2}$ we obtain all dyadic rational multiples of $x$ by using $\dotplus$.  It then follows by the Stone-Weierstrass Theorem for lattices (see \cite{Gillman1976}) that every continuous $F : [0, 1]^n \to [0, 1]$ can be uniformly approximated by functions on our list of connectives.
\end{rem}

\begin{rem}
The connectives we have chosen are the same as those used in \L ukasiewicz-Pavelka logic.  Pavelka added the rational constant connectives to \L ukasiewicz logic and proved a completeness theorem for the resulting \L ukasiewicz-Pavelka logic \cite{Pavelka1979, Pavelka1979a, Pavelka1979b}.  Later, H\'ajek, Paris, and Shepherdson proved that \L ukasiewicz-Pavelka logic is a conservative extension of \L ukasiewicz logic \cite{Hajek2000}.  The observations above show that the expressive power of \L ukasiewicz-Pavelka logic is the same as the expressive power of first-order continuous logic.
\end{rem}

In the setting of metric structures based on complete metric spaces, an infinitary $[0, 1]$-valued logic similar to ours was used in \cite{benYaacov2009}, but with additional technical requirements on the moduli of continuity of the $\phi_n$ when forming $\sup_{n<\omega}\phi_n$.

We think of $\sup_{n < \omega}\phi_n$ as an approximate infinitary disjunction of the $\phi_n$'s, since $\mc{M} \models \sup_{n < \omega}\phi_n$ if and only if for every $\epsilon \in \mathbb{Q} \cap (0, 1)$ there is $n < \omega$ such that $\phi_n^{\mc{M}} > \epsilon$.  We sometimes write $\bigvee_n \phi_n$ instead of $\sup_n \phi_n$.   Similarly, we think of $\sup_y\psi$ as an approximate version of $\exists y\,\psi$, and sometimes write $\exists y$ instead of $\sup_y$.  We define $\inf_n\phi_n$ (or $\bigwedge_n \phi_n$) as an abbreviation for $\neg\sup_n\neg\phi_n$, and define $\inf_y \psi$ (or $\forall y\,\psi$) as an abbreviation for $\neg\sup_y\neg\psi$.  Both of these abbreviations have the expected semantics.  Moreover, $\mc{M} \models \bigwedge_n \phi_n$ if and only if $\mc{M} \models \phi_n$ for every $n$, and similarly $\mc{M} \models \forall y\,\phi(y)$ if and only if $\mc{M} \models \phi(a)$ for every $a \in \mc{M}$.

The following is a partial list of classes of structures from analysis that can be axiomatized in our $\IL$:
\begin{itemize}
\item{
All classes of structures axiomatizable in finitary continuous logic.  In the signature of lattices the class of Banach lattices isomorphic to $L_p(\mu)$ for a fixed $1 \leq p < \infty$ and measure $\mu$ is axiomatizable, by results from \cite{Bretagnolle1965/1966}, \cite{Dacunha-Castelle1972}.  The class of Banach spaces isometric to $L_p(\mu)$ is also axiomatizable in the signature of Banach spaces (see \cite{Henson1976}), as is the class of Banach spaces isometric to $C(K)$ for a fixed compact Hausdorff space $K$ (see \cite{Heinrich1981}).  Further examples are described in \cite[Chapter 13]{Henson2003}.  More recent examples include subclasses of the class of Nakano spaces \cite{Poitevin2008}.
}
\item{
In any signature with countably many constants $(c_i)_{i < \omega}$, the statement that the constants form a dense set can be expressed by the following sentence:
\[\forall x \bigvee_{i < \omega}(d(x, c_i) = 0).\]
}
\item{
In the signature of normed spaces with countably many new constants $(c_i)_{i < \omega}$, the following formula $\phi(x)$ expresses that $x \in \overline{\op{span}}\set{c_i : i < \omega}$:
\[\phi(x) : \bigvee_{n < \omega} \bigvee_{a_0 \in \mathbb{Q} \cap (0, 1)} \cdots \bigvee_{a_{n-1} \in \mathbb{Q} \cap (0, 1)} \left(\norm{x - \sum_{i < n}a_ic_i} = 0\right).\]
We can express that $(c_i)_{i < \omega}$ is a $\lambda$-basic sequence for a fixed $\lambda$ with the sentence $\sigma_\lambda$:
\[\sigma_\lambda : \bigwedge_{N < \omega}\bigwedge_{a_0 \in \mathbb{Q} \cap (0, 1)} \cdots \bigwedge_{a_{N-1} \in \mathbb{Q} \cap (0, 1)}\left(\max_{n \leq N}\norm{\sum_{j=1}^n a_jc_j} \leq \lambda\norm{\sum_{j=1}^{N}a_jc_j}\right).\]
We can therefore express that $(c_i)_{i < \omega}$ is a Schauder basis:
\[\left(\forall x \phi(x)\right) \wedge \bigvee_{\lambda \in \mathbb{Q}}\sigma_\lambda.\]
Note that this cannot be expressed in the finitary fragment of $\IL$, since having a Schauder basis implies separability, and every separable Banach space is elementarily equivalent (in the finitary fragment) to a non-separable space.

The same ideas as in the above example allow us to express that $X$ (or equivalently, $X^*$) is not super-reflexive -- see \cite[Theorem 3.22]{Pisier}.
}
\item{
In the signature of normed spaces with an additional predicate $\vertiii{\cdot}$, we can express that $\norm{\cdot}$ and $\vertiii{\cdot}$ are equivalent by the axioms for $\vertiii{\cdot}$ being a norm, plus the sentence:
\[\bigvee_{C \in \mathbb{Q}} \bigvee_{D \in \mathbb{Q}} \forall x \left(C\norm{x} \leq \vertiii{x} \leq D\norm{x}\right).\]
}
\item{
In the signature of Banach spaces augmented with two new sorts $Y, Z$ for closed (infinite-dimensional) subspaces, the following expresses that $Y$ and $Z$ witness the failure of hereditary indecomposability (see \cite[Proposition 1.1]{Argyros2004}):
\[\bigvee_{\delta \in \mathbb{Q} \cap (0, 1)} \forall y \in Y \forall z \in Z \left(\norm{y-z} \geq \delta \norm{y+z}\right).\]
}
\item{
Failures of reflexivity can be expressed as follows.  Beginning with a two-sorted signature, each sort being the signature for Banach spaces, add countably many constants $(c_i)_{i < \omega}$ to the first sort, and $(c_i^*)_{i < \omega}$ to the second sort.  Let $S$ be the signature obtained by then adding a relation symbol $F$ for the natural pairing on $X \times X^*$.  Then in structures $(X, X^*)$, the following expresses that the constants witness the non-reflexivity of $X$ (see \cite[Theorem 3.10]{Pisier}):
\[\bigvee_{\theta \in \mathbb{Q} \cap (0, 1)} \bigwedge_{j < \omega}\left(\bigwedge_{i < j} F(c_i, c_j^*) = 0\right) \wedge \left(\bigwedge_{j \leq i < \omega} F(c_i, c_j^*) = \theta\right).\]

This example cannot be expressed in the finitary fragment of $\IL$, since it is known that there are reflexive Banach spaces with non-reflexive ultrapowers.
}
\item{
The failure of a Banach space to be stable, in the sense of Krivine and Maurey \cite{Krivine1981}, can be axiomatized in the signature of normed spaces with constants $(c_i)_{i < \omega}$ and $(d_i)_{i < \omega}$ as follows:
\[\bigvee_{\epsilon \in \mathbb{Q} \cap (0, 1)} \bigvee_{j < \omega} \bigvee_{i < j} \abs{\norm{c_i - d_j} - \norm{c_j - d_i}} \geq \epsilon.\]
More generally, we may replace $\norm{x-y}$ with any formula $\phi(x, y)$ to express that $\phi$ is not stable (see \cite{BenYaacov2010}).  It is well-known that stability is not axiomatizable in finitary logic.
}
\end{itemize}

\subsection{Fragments of $\IL$}
In the discrete setting the Omitting Types Theorem does not hold for $\IL$, but does hold for countable fragments of $\IL$.  Since our goal is to obtain a $[0, 1]$-valued version of the Omitting Types Theorem, we introduce the notion of a fragment of our $[0, 1]$-valued $\IL$.

\begin{defn}
Let $S$ be a metric signature.  A \defined{fragment} of $\IL(S)$ is a set $L$ of $\IL(S)$-formulas with the following properties:
\begin{itemize}
\item{
Every atomic formula is in $L$.
}
\item{
For each $q \in \mathbb{Q} \cap [0, 1]$, the constant formula $q$ is in $L$.
}
\item{
$L$ is closed under $\sup_x$, $\LI$, and $\neg$.
}
\item{
$L$ is closed under substituting terms for variables.
}
\end{itemize}
\end{defn}

It is easy to see that given any countably many formulas of $\IL$ there is a smallest countable fragment containing all of them.  Note also that every fragment contains every finitary formula.  If $L$ is a fragment of $\IL$, and $\mc{M}, \mc{N}$ are $L$-structures, we write $\mc{N} \equiv_L \mc{M}$ to mean $\sigma^{\mc{M}} = \sigma^{\mc{N}}$ for all $L$-sentences $\sigma$.  We write $\mc{N} \preceq_L \mc{M}$ if $\mc{N} \subseteq \mc{M}$ and for all $L$-formulas $\phi(x_1, \ldots, x_n)$ and all $a_1, \ldots, a_n \in \mc{N}$, $\phi^{\mc{N}}(a_1, \ldots, a_n) = \phi^{\mc{M}}(a_1, \ldots, a_n)$.

We will need the following version of the Downward L\"owenheim-Skolem Theorem.  The full logic $\IL$ does not satisfy Downward L\"owenheim-Skolem, but for countable fragments the standard proof from the first-order case adapts easily.

\begin{prop}[Downward L\"owenheim-Skolem]\label{propDownLS}
Let $L$ be a countable fragment of $\IL$, and let $\mc{M}$ be an $L$-structure.  Let $A \subseteq \mc{M}$ be a countable set of elements of $\mc{M}$.  Then there is a countable $L$-structure $\mc{N}$ such that $\mc{N} \preceq_L \mc{M}$ and $A \subseteq \mc{N}$.
\end{prop}

It will occasionally be useful for us to restrict the formulas in a fragment to those which define continuous functions on each structure.

\begin{defn}\label{defnFragmentCts}
A fragment $L$ of $\IL$ is \defined{continuous} if for every $L$-formula $\phi(x_1, \ldots, x_n)$ and every $L$-structure $\mc{M}$, the function from $\mc{M}^n$ to $[0, 1]$ defined by $\overline{a} \mapsto \phi^{\mc{M}}(\overline{a})$ is continuous.
\end{defn}

In \cite{benYaacov2009} syntactic conditions on formulas are presented that ensure continuity, so their infinitary logic is an example of a continuous fragment of our $\IL$.  It is easy to see that if $L$ is any continuous fragment then for every $L$-structure $\mc{M}$, $\mc{M} \preceq_L \overline{\mc{M}}$.  

\section{The logic topology}\label{SectionTopology}
We now describe a topological space associated with a given fragment of $\IL$.  A space similar to the one presented here was used by Caicedo and Iovino in \cite{Caicedo2012} to study omitting types in the context of abstract $[0, 1]$-valued logics, and for the finitary part of our $\IL$ in particular.

\begin{defn}
Let $S$ be a metric signature, and let $L$ be a fragment of $\IL(S)$.  Let $\op{Str}(L)$ denote the class of all countable $L$-structures.  For each $L$-theory $T$, define
\[\op{Mod}(T) = \set{\mc{M} \in \op{Str}(L) : \mc{M} \models T}.\]
\end{defn}

Observe that the collection of classes of the form $\op{Mod}(T)$ is closed under finite unions and arbitrary intersections (the latter uses that $L$ is closed under $\wedge$ and $\vee$).  Also, $\op{Str}(L) = \op{Mod}(\emptyset)$.  Hence the collection of complements of classes of the form $\op{Mod}(T)$ forms a topology on $\op{Str}(L)$.  This is the topological space that we will use to prove the Omitting Types Theorem.

\begin{defn}
The \defined{logic topology} on $\op{Str}(L)$ is the topology where the closed classes are exactly those of the form $\op{Mod}(T)$ for some $L$-theory $T$. 
\end{defn}

For any $L$-sentence $\sigma$, the function from $\op{Str}(L)$ to $[0, 1]$ defined by $\mc{M} \mapsto \sigma^{\mc{M}}$ is continuous.  This is because for each $r \in \mathbb{Q} \cap (0, 1)$ we have that $\sigma \leq r$ and $r \leq \sigma$ are $L$-sentences (see the discussion of the semantics of $\LI$ in Section~\ref{SectionInfinitary}).  We therefore have $\sigma^{-1}([r, s]) = \op{Mod}(\sigma \geq r \wedge \sigma \leq s)$ for every $r, s \in \mathbb{Q} \cap [0, 1]$.  It follows immediately from the definition of the logic topology that functions defined by sentences in this way are sufficient to separate points from closed classes.  As a result, $\op{Str}(L)$ is a completely regular space.

On the other hand, we also have that if $\mc{M}, \mc{N} \in \op{Str}(L)$, then $\mc{M}, \mc{N}$ are topologically indistinguishable in the logic topology if and only if $\mc{M} \equiv_L \mc{N}$.  As a result, the logic topology is not $T_0$.  It is possible to create a Hausdorff space that shares many of the properties of $\op{Str}(L)$ by taking the quotient of $\op{Str}(L)$ by the elementary equivalence relation, but for our purposes it is simpler to work directly with $\op{Str}(L)$ and its subspaces.

\begin{rem}
Our definition of $\op{Str}(S)$ raises certain foundational issues.  The logic topology is defined as a collection of proper classes, and thus is problematic from the point of standard axiomatizations of set theory, such as ZFC.  There are two natural ways to overcome this difficulty.  The first is to replace the class of all $L$-structures by the set of all complete $L$-theories.  Informally, this is equivalent to working with the quotient $\op{Str}(L) / \equiv_L$ mentioned above.  An alternative approach is to notice that every structure we need in the proof of Theorem~\ref{thmOmittingTypesStatement} is of cardinality at most $2^{\aleph_0}$.  We could then use Scott's trick (see e.g. \cite[9.3]{Jech1978}) to select one representative from each isomorphism class of $L$-structures of cardinality at most $2^{\aleph_0}$, and then replace the class of all $L$-structures by the set of these chosen representatives.  In what follows we will use $\op{Str}(L)$ as originally presented, as the reader will have no difficulty translating our arguments into either of these two approaches.
\end{rem}

\section{Proof of omitting types}\label{SectionOmittingTypes}
We fix, for the entirety of this section, a metric signature $S$ and a countable fragment $L$ of $\IL(S)$.  If $C$ is any set of new constant symbols, we denote by $L_C$ the smallest fragment of $\IL(S \cup C)$ containing $L$.  Note that if $C$ is countable then $L_C$ is a countable fragment of $\IL(S \cup C)$.  The sentences of $L_C$ are exactly those sentences of the form $\phi(c_1, \ldots, c_n)$ for some $c_1, \ldots, c_n \in C$ and $\phi(x_1, \ldots, x_n) \in L$.  If $D$ is a set of constant symbols with $C \subseteq D$ and $T$ is an $L_C$-theory, we write $\op{Mod}_{L_D}(T) = \set{\mc{M} \in \op{Str}(L_D) : \mc{M} \models T}$ and $\op{Mod}_{L_C}(T) = \set{\mc{M} \in \op{Str}(L_C) : \mc{M} \models T}$ when necessary to avoid ambiguity.  If $\mc{M}$ is an $L$-structure and $\overline{a} = \set{a_i : i < \omega}$ is a set of elements of $\mc{M}$, then the $L_C$ structure obtained from $\mc{M}$ by interpreting $c_i$ as $a_i$ is denoted by $\tup{\mc{M}, \overline{a}}$.

We now fix a countable set $C = \set{c_0, c_1, \ldots}$ of new constant symbols and an enumeration $\set{\phi_0(x), \phi_1(x), \ldots}$ of the $L_C$-formulas in exactly one free variable $x$.  We will primarily work in the following subspace of $\op{Str}(L_C)$:

\[\mc{W} = \bigcap_{i < \omega} \bigcap_{r \in \mathbb{Q} \cap (0, 1)} \left(\op{Mod}_{L_C}(\sup_x \phi_i(x) < 1) \cup \bigcup_{j < \omega}\op{Mod}_{L_C}(\phi_i(c_j) > r)\right).\]

The following remark states the main property of $\mc{W}$ that we will use.

\begin{rem}\label{remWProperty}
If $\tup{\mc{M}, \overline{a}} \in \mc{W}$ and $\mc{M} \models \sup_x\phi(x)$, then for each $\epsilon \in \mathbb{Q} \cap (0, 1)$ there is $j < \omega$ such that $\tup{\mc{M}, \overline{a}} \models \phi(c_j) \geq \epsilon$.  More generally, it follows from the fact that we can express inequalities in our formulas that if $(\sup_x\phi(x))^{\tup{\mc{M}, \overline{a}}} > r$ then there exists $r' \in \mathbb{Q} \cap (r, 1)$ and $j < \omega$ such that $\tup{\mc{M}, \overline{a}} \models \phi(c_j) \geq r'$.
\end{rem}

The preceding remark gives the following version of the Tarski-Vaught test for structures in $\mc{W}$.  The proof, which is a straightforward induction on the complexity of formulas, is left to the reader.

\begin{lem}\label{lemElementary}
If $\tup{\mc{M}, \overline{a}} \in \mc{W}$, then $\mc{M} \upharpoonright \tup{\overline{a}} \preceq_L \mc{M}$, where $\mc{M} \upharpoonright \tup{\overline{a}}$ is the substructure of $\mc{M}$ generated by $\overline{a}$.
\end{lem}

We note that $\mc{W}$ is non-empty, since given any countable $L$-structure $\mc{M}$ we may interpret $C$ as an enumeration $\overline{a}$ of $\mc{M}$ to obtain $\tup{\mc{M}, \overline{a}} \in \mc{W}$.

There are two parts to the proof of our main result.  First, in Section~\ref{subsectionComplete} we show that $\mc{W}$ is \v{C}ech-complete.  Then in Section~\ref{subsectionProof} we relate the model-theoretic notion of principal types to Baire category in $\mc{W}$, and use this to prove the Omitting Types Theorem.  Section~\ref{subsectionCompleteStructures} is concerned with adapting our proofs to the context of structures based on complete metric spaces.

\subsection{\v{C}ech-completeness of $\mc{W}$}\label{subsectionComplete}
Fix an enumeration $\set{\sigma_0, \sigma_1, \ldots}$ of the $L_C$-sentences such that $\sigma_0$ is an atomic sentence.  To prove that $\mc{W}$ is \v{C}ech-complete we must show that it has a complete sequence of open covers (see Definition~\ref{defnCechComplete}).  In fact there are many such sequences; the following lemma gives the existence of a sequence with the properties we will need.  By an \defined{open rational interval} in $[0, 1]$, we mean an interval $I \subseteq [0, 1]$ with rational endpoints that is open in the subspace topology on $[0, 1]$.

\begin{lem}\label{lemCompleteSequence}
There exists a sequence $\tup{\mc{U}_n : n < \omega}$ of open covers of $\mc{W}$ with the following properties:
\begin{enumerate}
\item{
For every $n$ and every $\epsilon > 0$ there is $l \geq n$ such that for each $U \in \mc{U}_l$ there is a rational open interval $I_U$ with $\op{length}(I_U) \leq \epsilon$ such that for all $\mc{N} \in U$, $\sigma_n^{\mc{N}} \in I_U$.
}
\item{
For every $n$, if $k \leq n$ is such that $\sigma_k = \sup_{i < \omega}\chi_i$, then for each $U \in \mc{U}_n$ there is a rational open interval $I$ in $[0, 1]$, and a $j < \omega$, such that for all $\mc{N} \in U$, $(\sup_i\chi_i)^{\mc{N}} \in I$ and $\chi_j^{\mc{N}} \in I$.
}
\item{
For every $n$, if $k \leq n$ is such that $\sigma_k = \sup_x\phi$, then for each $U \in \mc{U}_n$ there is a rational open interval $I$ in $[0, 1]$ and a $j < \omega$ such that for all $\mc{N} \in U$, $(\sup_x\phi)^{\mc{N}} \in I$ and $\phi(c_j)^{\mc{N}} \in I$.
}
\end{enumerate}
\end{lem}
\begin{proof}
We first define a sequence $(\mc{I}_n)_{n < \omega}$ of open covers of $[0, 1]$, the $n$th of which corresponds to splitting $[0, 1]$ into $n$ rational open intervals in $[0, 1]$ with small overlap.  To do this, for each $n < \omega$ let $\epsilon_n = \frac{1}{2^{n+2}}$.  For each $n$, define an open cover of $[0, 1]$ as follows:
\[\mc{I}_n = \set{\left[0, \frac{1}{n+2}+\epsilon_n\right), \left(\frac{1}{n+2} - \epsilon_n, \frac{2}{n+2}+\epsilon_n\right), \cdots, \left(\frac{n+1}{n+2}-\epsilon_n, 1\right]}.\]

For a sentence $\sigma$ and a rational open interval $I \subseteq [0, 1]$, we temporarily abuse notation to write
\[\op{Mod}(\sigma \in I) = \set{\mc{M} \in \mc{W} : \sigma^{\mc{M}} \in I}.\]

We construct the sequence $\tup{\mc{U}_n : n < \omega}$ recursively, so that the following properties hold:
\begin{enumerate}
\item[(a)]{
Each $\mc{U}_n$ is an open cover of $\mc{W}$,
}
\item[(b)]{
Each $U \in \mc{U}_n$ is of the form $U = \bigcap O_U$, where $O_U$ is a finite collection of open classes such that:
\begin{enumerate}
\item[(i)]{
Each element of $O_U$ is of the form $\op{Mod}(\theta \in J)$, where $\theta$ is a sentence and $J \in \mc{I}_n$,
}
\item[(ii)]{
For each $k \leq n$ there is $J_k \in \mc{I}_n$ such that $\op{Mod}(\sigma_k \in J_k) \in O_U$,
}
\item[(iii)]{
If $\op{Mod}\left(\sup_{i < \omega}\chi_i \in J\right) \in O_U$, then there exists $j < \omega$ such that $\op{Mod}(\chi_j \in J) \in O_U$,
}
\item[(iv)]{
If $\op{Mod}\left(\sup_x\phi\right) \in O_U$ then there exists $j < \omega$ such that $\op{Mod}(\phi(c_j) \in J) \in O_U$.
}
\end{enumerate}
}
\end{enumerate}
It is clear that a sequence $\tup{\mc{U}_n : n < \omega}$ satisfying (a) and (b) will satisfy (1) -- (3).

For the base case, define
\[\mc{U}_0 = \set{\op{Mod}(\sigma_0 \in I) : I \in \mc{I}_0}.\]
Since the intervals in $\mc{I}_0$ are open, $\mc{U}_0$ is an open cover, and the conditions in (b) are satisfied trivially.

Suppose that $\mc{U}_n$ is defined satisfying (a) and (b).  We first refine $\mc{U}_n$ to a cover $\widetilde{\mc{U}}_n$ as follows.  For each function $f : \mc{I}_n \to \mc{I}_{n+1}$, and each $U \in \mc{U}_n$, let $O_U^f = \set{\op{Mod}(\theta \in f(J)) : \op{Mod}(\theta \in J) \in O_U}$, and let $U^f = \bigcap O_U^f$.  Then let
\[\widetilde{\mc{U}}_n = \set{U^f : U \in \mc{U}_n, f : \mc{I}_n \to \mc{I}_{n+1}}.\]
If $\sigma_{n+1}$ is not an infinitary disjunction and is not of the form $\sup_x\phi$, then define
\[\mc{U}_{n+1} = \set{U \cap \op{Mod}(\sigma_{n+1} \in I) : U \in \widetilde{\mc{U}}_n, I \in \mc{I}_{n+1}}.\]
Note that $\mc{U}_{n+1}$ is a cover of $\mc{W}$ since $\widetilde{\mc{U}}_n$ is a cover of $\mc{W}$ and $\mc{I}_{n+1}$ is a cover of $[0, 1]$.
If $\sigma_{n+1}$ is the infinitary disjunction $\sup_{i < \omega}\chi_i$, then define
\[\mc{U}_{n+1} = \set{U \cap \op{Mod}(\sigma_{n+1} \in I) \cap \op{Mod}(\chi_j \in I) : U \in \widetilde{\mc{U}}_n, I \in \mc{I}_{n+1}, j < \omega}.\]
Finally, if $\sigma_{n+1}$ is of the form $\sup_x\phi$, define
\[\mc{U}_{n+1} = \set{U \cap \op{Mod}(\sigma_{n+1} \in I) \cap \op{Mod}(\phi(c_j) \in I) : U \in \widetilde{\mc{U}}_n, I \in \mc{I}_{n+1}, j < \omega}.\]

It is easy to see that (b) is preserved, so we only need to observe that $\mc{U}_{n+1}$ is a cover of $\mc{W}$.  This follows from Remark~\ref{remWProperty} and the fact that $\widetilde{\mc{U}}_n$ is a cover.
\end{proof}

\begin{prop}\label{propCechComplete}
The space $\mc{W}$ is \v{C}ech-complete.
\end{prop}
\begin{proof}
Let $\tup{\mc{U}_n : n < \omega}$ be a sequence of open covers as given by Lemma~\ref{lemCompleteSequence}.  Let $\mc{F}$ be a centred family of closed sets such that for each $n < \omega$ there is $F_n \in \mc{F}$ and $U_n \in \mc{U}_n$ such that $F_n \subseteq U_n$.  To show that $\tup{\mc{U}_n : n < \omega}$ is a complete sequence of open covers, we must show that $\bigcap \mc{F} \neq \emptyset$.  It is easy to check, using (1) from Lemma~\ref{lemCompleteSequence}, that $\bigcap \mc{F} = \bigcap_{n < \omega}F_n$.

For each $n < \omega$, choose $\mc{M}_n \in F_0 \cap \cdots \cap F_n$.  Let $\mc{D}$ be a non-principal ultrafilter on $\omega$.  We will show that $\prod_{\mc{D}}\mc{M}_n \in \mc{W} \cap \bigcap_{n < \omega}F_n$.

\begin{claim}\label{Claim3}
For any $L_C$-sentence $\sigma$, $\sigma^{\prod_{\mc{D}}\mc{M}_n} = \lim_{n \to \mc{D}}\sigma^{\mc{M}_n}$.
\end{claim}
\begin{proof}[Proof of Claim~\ref{Claim3}]  
The proof is by induction on the complexity of $\sigma$.  The case where $\sigma$ is an atomic sentence follows directly from the definition of the ultraproduct (Definition~\ref{defnUltraproduct}), and the case where $\sigma$ is the result of applying a finitary connective follows from the continuity of the finitary connectives and the definition of ultrafilter limits, so we only need to deal with the infinitary disjunction and $\sup_x\phi$ cases.  
\newline\newline\noindent
$\underline{\sigma = \sup_{i < \omega}\chi_i}$:\\
It is sufficient to show that for each $a \in \mathbb{Q} \cap (0, 1)$, $a < \sigma^{\prod_{\mc{D}}\mc{M}_n}$ if and only if $\set{n < \omega : \sigma^{\mc{M}_n} > a} \in \mc{D}$.

Suppose $a < \sigma^{\prod_{\mc{D}}\mc{M}_n}$.  Then 
\[\sup_{i < \omega}\chi_{i}^{\prod_{\mc{D}}\mc{M}_n} > a.\]  
Hence there is some $j < \omega$ such that
\[\chi_{j}^{\prod_{\mc{D}}\mc{M}_n} > a.\]
So by the inductive hypothesis, $\lim_{n\to\mc{D}}\chi_{j}^{\mc{M}_n} > a$.  That is,
\[\set{n < \omega : \chi_{j}^{\mc{M}_n} > a} \in \mc{D}.\]
We have $\chi_{j}^{\mc{M}_n} \leq \sigma^{\mc{M}_n}$ for each $n$, so $\set{n < \omega : \sigma^{\mc{M}_n} > a} \supseteq \set{n < \omega : \chi_{j}^{\mc{M}_n} > a}$, and hence
\[\set{n < \omega : \sigma^{\mc{M}_n} > a} \in \mc{D}.\]

Now assume $\set{n < \omega : \sigma^{\mc{M}_n} > a} \in \mc{D}$.  Note that, by the inductive hypothesis, it suffices to find $j < \omega$ such that $\set{n < \omega : \chi_{j}^{\mc{M}_n} > a} \in \mc{D}$.  Find $l < \omega$ such that $\sigma = \sigma_l$.  Find $k \geq l$ such that $\sigma^{\mc{M}_k} > a$ and for all $\mc{N} \in U_k$, $\sigma^{\mc{N}} > a$ (by (1) of Lemma~\ref{lemCompleteSequence}).  By (2) of Lemma~\ref{lemCompleteSequence}, there is some $j < \omega$ such that for all $\mc{N} \in U_k$, $\chi_{j}^{\mc{N}} > a$.  In particular, for all $n \geq k$, $\chi_{j}^{\mc{M}_n} > a$.  Thus for cofinitely many $n$ we have $\chi_{j}^{\mc{M}_n} > a$, and $j$ is as desired.
\newline\newline\noindent
$\underline{\sigma = \sup_x\phi(x)}$.\\
Suppose that $\set{n < \omega : \left(\sup_x\phi\right)^{\mc{M}_n} > a} \in \mc{D}$.  As in the previous case, by (1) of Lemma~\ref{lemCompleteSequence} we can find $k < \omega$ such that $(\sup_x\phi)^{\mc{N}} > a$ for all $\mc{N} \in U_k$.  By (3) of Lemma~\ref{lemCompleteSequence} we get $j < \omega$ such that $\phi(c_j)^{\mc{N}} > a$ for all $\mc{N} \in U_k$.  For all $n \geq k$ we have $\phi(c_j)^{\mc{M}_n} > a$, and hence $\lim_{n \to \mc{D}}\phi(c_j)^{\mc{M}_n} > a$.  By the inductive hypothesis we have $\phi(c_j)^{\prod_{\mc{D}}\mc{M}_n} > a$, and therefore $\left(\sup_x\phi\right)^{\prod_{\mc{D}}\mc{M}_n} > a$ as well.

Now suppose that $\set{n < \omega : \left(\sup_x\phi\right)^{\mc{M}_n} > a} \not\in \mc{D}$.  In order to prove that $\left(\sup_x\phi\right)^{\prod_{\mc{D}}\mc{M}_n} \leq a$, we consider two cases.  The case $\set{n < \omega : \left(\sup_x\phi\right)^{\mc{M}_n} < a} \in \mc{D}$ is handled in the same way as the previous paragraph.  For the other case, suppose that $\set{n < \omega : \left(\sup_x\phi\right)^{\mc{M}_n} = a} \in \mc{D}$.  Then for each $\epsilon \in \mathbb{Q} \cap (0, 1)$ such that $\epsilon < \min\set{a, 1-a}$, we also have 
\[\set{n < \omega : \left(\sup_x\phi\right)^{\mc{M}_n} \in (a-\epsilon, a+\epsilon)} \in \mc{D}.\] 
As in the preceding cases, this implies that $\left(\sup_x\phi\right)^{\prod_{\mc{D}}\mc{M}_n} \in (a-\epsilon, a+\epsilon)$ for each such $\epsilon$.  Taking $\epsilon \to 0$ we obtain $\left(\sup_x\phi\right)^{\prod_{\mc{D}}\mc{M}_n} = a$.

\renewcommand{\qedsymbol}{$\dashv$ -- Claim~\ref{Claim3}}
\end{proof}

For each $F \in \mc{F}$, let $T_F$ be a theory such that $F = \op{Mod}(T_F)$.  Then Claim~\ref{Claim3} implies that $\prod_{\mc{D}}\mc{M}_n \models T_{F_m}$ for every $m < \omega$, so it only remains to check that $\prod_{\mc{D}}\mc{M}_n \in \mc{W}$.  The proof is essentially the same as the last case of the claim.  Suppose that $\phi(x)$ is an $L_C$-formula in one free variable, and that ${\left(\sup_x\phi\right)^{\prod_{\mc{D}}\mc{M}_n}=1}$.  Fix $r \in \mathbb{Q} \cap (0, 1)$.  We need to find $j$ such that $\phi(c_j)^{\prod_{\mc{D}}\mc{M}_n} > r$.  By Claim~\ref{Claim3} we have $\lim_{n \to \mc{D}}\left(\sup_x\phi\right)^{\mc{M}_n} = 1$, so 
\[\set{n < \omega : \left(\sup_x\phi\right)^{\mc{M}_n} > r} \in \mc{D}.\]
Using (1) and (3) of Lemma~\ref{lemCompleteSequence} we can find $k$ and $j$ such that $\phi(c_j)^{\mc{N}} > r$ for all $\mc{N} \in U_k$.  Hence $\set{n < \omega : \phi(c_j)^{\mc{M}_n} > r} \in \mc{D}$, and by Claim~\ref{Claim3} we have $\phi(c_j)^{\prod_{\mc{D}}\mc{M}_n} > r$.
\end{proof}

\begin{cor}\label{lemWBaire}
Let $T$ be a consistent $L$-theory.  Then $\mc{W} \cap \op{Mod}_{L_C}(T)$ is non-empty and is Baire.
\end{cor}
\begin{proof}
Since $T$ is consistent it has a countable model $\mc{M}$, by Downward L\"owenheim-Skolem (Proposition~\ref{propDownLS}).  If $\overline{a}$ is an enumeration of $\mc{M}$, then $\tup{\mc{M}, \overline{a}} \in \mc{W} \cap \op{Mod}_{L_C}(T) \neq \emptyset$.  The fact that $\mc{W} \cap \op{Mod}_{L_C}(T)$ is Baire follows immediately from Lemma~\ref{lemCompleteProperties} and Proposition~\ref{propCechComplete}.
\end{proof}

\subsection{Proof of Omitting Types}\label{subsectionProof}
In this section we connect the model-theoretic notions in the statement of Theorem~\ref{thmOmittingTypesStatement} to the topology of the space $\mc{W}$.  The connection between Baire spaces and the Omitting Types Theorem in classical logic is well-known.  We give a proof in our $[0, 1]$-valued setting for completeness, following the arguments in \cite{Caicedo2012} closely.  For simplicity we present the details of the proof in the case where the signature $S$ contains no function symbols.  After the proof is complete we will describe the modifications necessary for the general case.  

Recall that a set of $L$-formulas $\Sigma(x_1, \ldots, x_n)$ is a \defined{type of $T$} if there is $\mc{M} \models T$ and $a_1, \ldots, a_n \in \mc{M}$ such that $\mc{M} \models \phi(a_1, \ldots, a_n)$ for all $\phi(x_1, \ldots, x_n) \in \Sigma(x_1, \ldots, x_n)$.  When $S$ has no function symbols, the definition of a type of $T$ being principal takes the following simplified form:

\begin{defn}\label{defnPrincipal1}
Let $T$ be an $L$-theory in a signature without function symbols.  A type $\Sigma(\overline{x})$ of $T$ \defined{principal} over $T$ is there is an $L$-formula $\phi(\overline{x})$ such that $T \cup \phi(\overline{x})$ is satisfiable, and for some $r \in \mathbb{Q} \cap (0, 1)$ we have $T \cup \set{\phi(\overline{x}) \geq r} \models \Sigma(\overline{x})$.  We say that such $\phi$ and $r$ \defined{witness} the principality of $\Sigma$.
\end{defn}

We can now give the connection between principality of types and Baire spaces.

\begin{lem}\label{lemPrincipalInterior}
Let $\Sigma(\overline{x})$ be a type of an $L$-theory $T$, and let $\overline{c}$ be new constant symbols.  Then $\Sigma(\overline{x})$ is principal if and only if $\op{Mod}_{L_{\overline{c}}}(T \cup \Sigma(\overline{c}))$ has nonempty interior in $\op{Mod}_{L_{\overline{c}}}(T)$.
\end{lem}
\begin{proof}
Assume that $\Sigma(\overline{x})$ is principal, and let $\phi(\overline{x}) \in L$ and $r \in \mathbb{Q} \cap (0, 1)$ witness the principality of $\Sigma$.  Then $T \cup \set{\phi(\overline{x})}$ is satisfiable, and hence $\op{Mod}_{L_{\overline{c}}}(T \cup \phi(\overline{c})) \neq \emptyset$.  If $r' \in \mathbb{Q} \cap (r, 1)$, then $\op{Mod}_{L_{\overline{c}}}(T) \cap \op{Mod}_{L_{\overline{c}}}(\phi(\overline{c}) > r')$ is a nonempty open subclass of $\op{Mod}_{L_{\overline{c}}}(T \cup \Sigma(\overline{c}))$.

Conversely, suppose that $\op{Mod}_{L_{\overline{c}}}(T \cup \Sigma(\overline{c}))$ has nonempty interior in $\op{Mod}_{L_{\overline{c}}}(T)$, so it contains a basic open class.  That is, there is an $L_{\overline{c}}$-sentence $\phi(\overline{c})$ such that 
\[\emptyset \neq \op{Mod}_{L_{\overline{c}}}(T) \cap \op{Mod}_{L_{\overline{c}}}(\phi(\overline{c}) > 0) \subseteq \op{Mod}_{L_{\overline{c}}}(T \cup \Sigma(\overline{c})).\]
It follows that there exists $s \in \mathbb{Q} \cap (0, 1)$ such that $T \cup \set{\phi(\overline{x}) \geq s}$ is satisfiable.  Our choices of $\phi$ and $s$ give us that
\[T \cup \set{\phi(\overline{x}) \geq s} \models T \cup \set{\phi(\overline{x}) > 0} \models \Sigma(\overline{x}).\]
It is easy to check that if $r \in \mathbb{Q} \cap (0, s)$ then the formula $s \to \phi$ and the rational $1 - r$ witness that $\Sigma$ is principal.
\end{proof}

\begin{lem}\label{lemRTi}
Let $T$ be an $L$-theory.  For any $\mb{i} = \tup{i_1, i_2, \ldots, i_n} \in \omega^{<\omega}$, let $R_{T, \mb{i}} : \mc{W} \cap \op{Mod}_{L_C}(T) \to \op{Mod}_{L_{\set{c_{i_1}, \ldots, c_{i_n}}}}(T)$ be the natural projection defined by
\[\tup{\mc{M}, \overline{a}} \mapsto \tup{\mc{M}, a_{i_1}, \ldots, a_{i_n}}.\]
Then $R_{T, \mb{i}}$ is continuous, open, and surjective.
\end{lem}
\begin{proof}
To keep the notation as simple as possible, we will give the proof only in the case where $\mb{i} = \tup{0}$ -- the general case is similar.  To see that $R_{T, \mb{i}}$ is continuous, observe that if $\sigma$ is any $L_{c_0}$-sentence then $\sigma$ is also an $L_C$-sentence, and the pre-image of the basic closed class $\op{Mod}_{L_{c_0}}(\sigma)$ under $R_{T, \mb{i}}$ is the closed class $\op{Mod}_{L_C}(\sigma)$.

Now suppose that $\phi(c_0, \ldots, c_m)$ is an $L_C$-sentence (with possibly some of the $c_i$'s, including $c_0$, not actually appearing).  Define the $L_{c_0}$-sentence $\theta(c_0)$ by
\[\inf_{x_1} \cdots \inf_{x_m} \phi(c_0, x_1, \ldots, x_m).\]
To finish the proof it suffices to show that $R_{T, \mb{i}}$ maps $\left(\mc{W} \cap \op{Mod}_{L_C}(T)\right)\setminus\op{Mod}_{L_C}(\phi(c_0, \ldots, c_m))$ onto $\op{Mod}_{L_{c_0}}(T) \setminus \op{Mod}_{L_{c_0}}(\theta(c_0))$.

Suppose that $\tup{\mc{M}, \overline{a}} \in \left(\mc{W} \cap \op{Mod}_{L_C}(T)\right)\setminus\op{Mod}_{L_C}(\phi(c_0, \ldots, c_m))$.  Then \\$\tup{\mc{M}, \overline{a}}~\not\models~\phi(c_0, \ldots, c_m)$, so there is $r \in \mathbb{Q} \cap (0, 1)$ such that 
\[\tup{\mc{M}, \overline{a}} \models \phi(c_0, \ldots, c_m) \leq r.\]
Then clearly
\[\tup{\mc{M}, a_0} \models \theta(c_0) \leq r.\]
It follows that $\tup{\mc{M}, a_0} \in \op{Mod}_{L_{c_0}}(T) \setminus \op{Mod}_{L_{c_0}}(\theta(c_0))$.

Now suppose that $\tup{\mc{M}, a_0} \in \op{Mod}_{L_{c_0}}(T) \setminus \op{Mod}_{L_{c_0}}(\theta(c_0))$.  As above, find $r \in \mathbb{Q} \cap (0, 1)$ such that $\tup{\mc{M}, a_0}~\models~\theta(c_0) \leq r$, and pick $r' \in (r, 1)$.  Then by definition of $\theta$ there are elements $a_1, \ldots, a_m \in \mc{M}$ such that
\[\tup{\mc{M}, a_0, a_1, \ldots, a_m} \models \phi(c_0, c_1, \ldots, c_m) \leq r'.\]
By Downward L\"owenheim-Skolem (Proposition~\ref{propDownLS}) we can find a countable $\mc{M}_0 \preceq_L \mc{M}$ containing $a_0, a_1, \ldots, a_m$.  Using the remaining constant symbols to enumerate $\mc{M}_0$ as $\overline{a}$, we have
\[\tup{\mc{M}, \overline{a}} \in \left(\mc{W} \cap \op{Mod}_{L_C}(T)\right)\setminus\op{Mod}_{L_C}(\phi(c_0, \ldots, c_m)),\] and $R_{T, \mb{i}}(\tup{\mc{M}, \overline{a}}) = \tup{\mc{M}, a_0}$.
\end{proof}

We now have all of the ingredients necessary to prove our main result.

\begin{thm}[Omitting Types]\label{thmOmittingTypes}
Let $T$ be a satisfiable $L$-theory and let $\set{\Sigma_j(\overline{x}_j)}_{j < \omega}$ be a countable set of types of $T$ that are not principal over $T$.  Then there is a model of $T$ that omits each $\Sigma_j$.
\end{thm}
\begin{proof}
For each $j < \omega$, write $\overline{x}_j = (x_0, \ldots, x_{n_j-1})$.  Then for $\mb{i} \in \omega^{n_j}$, define
\[\mc{C}_{T, j, \mb{i}} = R_{T, \mb{i}}^{-1}\left(\op{Mod}_{L_{\set{c_{i_0}, \ldots, c_{i_{n_j-1}}}}}(T \cup \Sigma_j(c_{i_0}, \ldots, c_{i_{n_j-1}}))\right) \subseteq \mc{W} \cap \op{Mod}_{L_C}(T).\]
By Lemmas~\ref{lemPrincipalInterior} and~\ref{lemRTi}, each $\mc{C}_{T, j, \mb{i}}$ is closed with empty interior.  Hence $\bigcup_{j < \omega, \mb{i} \in \omega^{n_j}}\mc{C}_{T, j, \mb{i}}$ is meagre in $\mc{W} \cap \op{Mod}_{L_C}(T)$.  Since $\mc{W} \cap \op{Mod}_{L_C}(T)$ is Baire (Lemma~\ref{lemWBaire}), there exists
\[\tup{\mc{M}, \overline{a}} \in \left(\mc{W} \cap \op{Mod}_{L_C}(T)\right) \setminus \bigcup_{j < \omega, \mb{i} \in \omega^{n_j}}\mc{C}_{T, j, \mb{i}}.\]
For such $\tup{\mc{M}, \overline{a}}$ we have by definition of the $\mc{C}_{T, j, \mb{i}}$'s that for every $j<\omega$ no subset of $\overline{a}$ is a realization of $\Sigma_j$.  Since we are in the case where there are no function symbols, $\overline{a}$ is the universe of a structure $\mc{M}_0$.  By Lemma~\ref{lemElementary}, $\mc{M}_0 \preceq_L \mc{M}$.  Thus $\mc{M}_0 \models T$ and omits every $\Sigma_j$.
\end{proof}

The preceding proof generalizes in a straightforward way to the case where the signature contains function symbols, but it is necessary to give a stronger definition of principal type.  The only difficulty is that when there are function symbols present not every subset of a structure is the universe of a substructure, so in the proof of Theorem~\ref{thmOmittingTypes} we need to take $\mc{M}_0$ to be $\mc{M} \upharpoonright \tup{\overline{a}}$.  The proof of Lemma~\ref{lemElementary} works even with function symbols present, so we still have that $\mc{M}_0 \preceq \mc{M}$, but we now need to prove that no subset of $\mc{M} \upharpoonright \tup{\overline{a}}$ realizes any of the $\Sigma_j$.  To do this, we introduce terms into the definition of principality.

\begin{defn}\label{defnPrincipal2}
Let $T$ be an $L$-theory.  A type $\Sigma(\overline{x})$ of $T$ \defined{principal} over $T$ is there is an $L$-formula $\phi(\overline{x})$, terms $t_1(\overline{y}), \ldots, t_n(\overline{y})$ (where $n$ is the length of $\overline{x}$), and $r \in \mathbb{Q} \cap (0, 1)$ such that the following hold:
\begin{itemize}
\item{
$T \cup \set{\phi(\overline{y})}$ is satisfiable, and
}
\item{
$T \cup \set{\phi(\overline{y}) \geq r} \models \Sigma(t_1(\overline{y}), \ldots, t_n(\overline{y}))$.
}
\end{itemize}
\end{defn}

The modification of principality to include terms was used by Keisler and Miller \cite{Keisler2001} in the context of discrete logic without equality, and independently by Caicedo and Iovino \cite{Caicedo2012} for $[0, 1]$-valued logic.  Taking Definition~\ref{defnPrincipal2} as the definition of principality, we may assume that whenever $\Sigma(\overline{x})$ is a type we wish to omit, and $t_1(\overline{y}), \ldots, t_n(\overline{y})$ are terms, then $\Sigma(t_1(\overline{y}), \ldots, t_n(\overline{y}))$ is also one of the types to be omitted.  Then we have that no subset of $\mc{M} \upharpoonright \tup{\overline{a}}$ realizes any of the types we wish to omit since elements of $\mc{M} \upharpoonright \tup{\overline{a}}$ are obtained from $\overline{a}$ by applying terms. \hfill\qed\newline

\begin{rem}
By assuming additional set-theoretic axioms it is possible to extend Theorem~\ref{thmOmittingTypes} to allow a collection of fewer than $2^{\aleph_0}$ non-principal types to be omitted.  To do this, observe that $\mc{W}$ has a countable base, so Martin's Axiom restricted to countable partial orders implies that $\mc{W}$ is $2^{\aleph_0}$-Baire (see Remark~\ref{remSetTheoretic}).  Then the same proof as above can be applied to a collection of fewer than $2^{\aleph_0}$ non-principal types.  If $T$ is a theory in a countable fragment $L$ of $\IL$ then there are at most $2^{\aleph_0}$ types of $T$.  There are theories in which every model realizes a non-isolated type (see \cite[Example 2.3.1]{Chang1990} for an example in the discrete case), so it is not generally possible to omit $2^{\aleph_0}$ non-prinicipal types.  Thus under the Continuum Hypothesis it is not always possible to omit $\aleph_1$ non-principal types.  These observations show that the extension of Theorem~\ref{thmOmittingTypes} to omitting $\aleph_1$ non-principal types is undecidable on the basis of ZFC.
\end{rem}
\subsection{Omitting Types in Complete Structures}\label{subsectionCompleteStructures}
In applications of $[0, 1]$-valued logics it is sometimes desirable to be able to produce metric structures based on \emph{complete} metric spaces.  There are two issues that need to be addressed in order to be able to take the metric completion of the structure obtained from Theorem~\ref{thmOmittingTypes}.  First, there are some types that may be omitted in a structure but not in its metric completion (such as the type of the limit of a non-convergent Cauchy sequence), so we need a stronger notion of principal type.  Second, because of the infinitary connectives, it may not be the case that every structure is elementarily equivalent to its metric completion.  

To resolve the first issue, we use the notion of \defined{metrically principal} types from \cite{Caicedo2012}.  If $\Sigma(x_1, \ldots, x_n)$ is a type, then for each $\delta \in \mathbb{Q} \cap (0, 1)$ we define:
\[\Sigma^\delta = \set{\sup_{y_1} \ldots \sup_{y_n} \left(\bigwedge_{k \leq n} d(x_k, y_k) \leq \delta \wedge \sigma(y_1, \ldots, y_n)\right) : \sigma \in \Sigma}.\]
We think of $\Sigma^\delta$ as a thickening of $\Sigma$, since if $\mc{M}$ is a structure and $a_1, \ldots, a_n \in \mc{M}$ realize $\Sigma$, then every $n$-tuple in the closed $\delta$-ball around $(a_1, \ldots, a_n)$ realizes $\Sigma^\delta$.

\begin{defn}
Let $L$ be a fragment of $\IL$ and let $T$ be an $L$-theory.  We say that a type $\Sigma(\overline{x})$ of $T$ is \defined{metrically principal} over $T$ if for every $\delta > 0$ the type $\Sigma^\delta(\overline{x})$ is principal over $T$.
\end{defn}

\begin{prop}\label{propMetricOmitting}
Let $L$ be a countable fragment of $\IL$, and let $T$ be a satisfiable $L$-theory.  For each $n < \omega$, suppose that $\Sigma_n$ is a type that is not metrically principal.  Then there is $\mc{M} \models T$ such that the metric completion of $\mc{M}$ omits each $\Sigma_n$.
\end{prop}
\begin{proof}
For each $n < \omega$, let $\delta_n > 0$ be such that $\Sigma_n^{\delta_n}$ is non-principal.  Using Theorem~\ref{thmOmittingTypes} we get $\mc{M} \models T$ that omits each $\Sigma_n^{\delta_n}$.  Fix $n < \omega$; we show that $\overline{\mc{M}}$, the metric completion of $\mc{M}$, omits $\Sigma_n$.  Suppose otherwise, and let $\overline{a} \in \overline{\mc{M}}$ be a realization of $\Sigma_n$ in $\overline{\mc{M}}$.  By definition of the metric completion there are $\overline{a}_1, \overline{a}_2, \ldots$ from $\mc{M}$ converging (coordinatewise) to $\overline{a}$.  For $k$ sufficiently large we then have that $\overline{a}_k$ is in the $\delta_n$-ball around $\overline{a}$.  As we observed earlier, this implies that $\overline{a}_k$ satisfies $\Sigma_n^{\delta_n}$, contradicting that $\Sigma_n^{\delta_n}$ is not realized in $\mc{M}$.
\end{proof}

The final problem to be resolved in order to have a satisfactory Omitting Types Theorem for complete structures is that we may not have $\mc{M} \equiv_L \overline{\mc{M}}$.  This problem arises because if $\phi(x)$ is a formula of $\IL$ and $\mc{M}$ is a structure, then the function from $\mc{M}$ to $[0, 1]$ given by $a \mapsto \phi^{\mc{M}}(a)$ may not be continuous.  Recall that a fragment $L$ of $\IL(S)$ is \defined{continuous} if $a \mapsto \phi^{\mc{M}}(a)$ is a continuous function for every $S$-structure $\mc{M}$ and every $L$-formula $\phi$ (Definition~\ref{defnFragmentCts}).  Applying Proposition~\ref{propMetricOmitting} we therefore have:

\begin{thm}[Omitting Types for Complete Structures]
Let $L$ be a countable continuous fragment of $\IL$, and let $T$ be a satisfiable $L$-theory.  For each $n < \omega$ let $\Sigma_n$ be a type that is not metrically principal.  Then there is $\mc{M} \models T$ such that $\mc{M}$ is based on a complete metric space and $\mc{M}$ omits each $\Sigma_n$.
\end{thm}

\section{Applications}\label{SectionApplication}
In this section we apply the Omitting Types Theorem to obtain a $[0, 1]$-valued version of Keisler's two-cardinal theorem (see \cite[Theorem 30]{Keisler1971}).  We will then apply the two-cardinal theorem to strengthen a result of Ben Yaacov and Iovino \cite{benYaacov2009} related to separable quotients of Banach spaces.

We begin with an easy lemma about metric spaces.

\begin{lem}\label{lemApplication}
Let $(M, d)$ be a metric space of density $\lambda$, where $\op{cof}(\lambda) > \omega$.  Then there is $R \in \mathbb{Q} \cap (0, 1)$ and a set $D \subseteq M$ with $\abs{D} = \lambda$ such that for all $x, y \in D$, $d(x, y) \geq R$, and for all $x \in M$ there exists $y \in D$ with $d(x, y) < R$.
\end{lem}
\begin{proof}
Build a sequence $\set{x_\alpha : \alpha < \lambda}$ in $M$ recursively, starting from an arbitrary $x_0 \in M$.  Given $\set{x_\alpha : \alpha < \beta}$, with $\beta < \lambda$, we have that $\set{x_\alpha : \alpha < \beta}$ is not dense in $M$.  Hence there exists $x_\beta \in M$ and $R_\beta \in \mathbb{Q} \cap (0, 1)$ such that $d(x_\beta, x_\alpha) \geq R_\beta$ for all $\alpha < \beta$.  Then since $\op{cof}(\lambda) > \omega$ there is $R \in \mathbb{Q} \cap (0, 1)$ and $S \in [\lambda]^{\lambda}$ such that $R = R_\alpha$ for every $\alpha \in S$.  Then $D = \set{x_\alpha : \alpha \in S}$ can be extended to the desired set.
\end{proof}

It will be important for us that certain predicates take values only in $\{0, 1\}$, and that this can be expressed in our logic.  For any formula $\phi(x)$, we define the formula $\Disc(\phi)$ to be $\phi(x) \vee \neg\phi(x)$.  It is clear that if $\mc{M} \models \forall x \Disc(\phi(x))$, then $\phi^{\mc{M}}(a) \in \{0, 1\}$ for every $a \in \mc{M}$; in this case we say that $\phi$ is \defined{discrete} in $\mc{M}$.  Note that if $\phi(x)$ is discrete in models of a theory $T$ then we can relativize quantifiers to $\set{x : \phi(x) = 1}$ in models of $T$.  We emphasize that discreteness of $\phi$ only means that $\phi$ takes values in $\{0, 1\}$, \emph{not} that the metric is discrete on $\set{x : \phi(x) = 1}$.

\begin{defn}
If $S$ is a metric signature with a distinguished unary predicate $U$, and $\kappa, \lambda$ are infinite cardinals, then we say that an $S$-structure $\mc{M} = \tup{M, U, \ldots}$ is \defined{of type $(\kappa, \lambda)$} if the density of $M$ is $\kappa$ and the density of $\set{a \in M : U(a) = 1}$ is $\lambda$.
\end{defn}

\begin{thm}\label{thmKeisler}
Let $S$ be a metric signature with a distinguished unary predicate symbol $U$, and let $L$ be a countable fragment of $\IL(S)$.  Let $T$ be an $L$-theory such that $T \models \forall x \Disc(U(x))$, and let $\mc{M} = \tup{M, V, \ldots}$ be a model of $T$ of type $(\kappa, \lambda)$ where $\kappa > \lambda \geq \aleph_0$.  Then there is a model $\mc{N} = \tup{N, W, \ldots} \equiv_L \mc{M}$ of type $(\aleph_1, \aleph_0)$.  Moreover, there is a model $\mc{M}_0 = \tup{M_0, V_0, \ldots}$ such that $\mc{M}_0 \preceq_L \mc{M}, \mc{M}_0 \preceq_L \mc{N}$, and $V_0$ is dense in $W$.
\end{thm}
\begin{proof}
By Downward L\"owenheim-Skolem, we may assume that $\mc{M}$ is of type $(\kappa^+, \kappa)$ for some $\kappa \geq \aleph_0$.  Our first step is to expand $\mc{M}$ into a structure in a larger language that includes an ordering of a dense subset of $\mc{M}$ in type $\kappa^+$.  To do this we expand the signature $S$ to a new signature $S'$ by adding a unary predicate symbol $L$, a binary predicate symbol $\sqeq$, a constant symbol $c$, and a unary function symbol $f$.  Let $M'$ be the disjoint union of $M$ and $\kappa^+$.  Extend the metric $d$ from $M$ to a metric $d'$ on $M'$ by making $d'$ the discrete metric on $\kappa^+$ and setting $d'(m, \alpha) = 1$ for every $m \in M, \alpha \in \kappa^+$.  We interpret $L$ as $L(x) = 1$ if and only if $x \in \kappa^+$.  Interpret $c$ as $\kappa$, and let $\sqeq$ be the characteristic function of the ordinal ordering on $\kappa^+$, and arbitrary elsewhere. 

Find $D \subseteq M$ of size $\kappa^+$, and $R \in \mathbb{Q} \cap (0, 1)$, as in Lemma~\ref{lemApplication}.  Define $f : M' \to M'$ so that below $\kappa$ the function $f$ is an enumeration of a dense subset of $V$, from $\kappa$ to $\kappa^+$ $f$ is an enumeration of $D$, and $f$ is arbitrary otherwise.  This gives a metric structure $\mc{M}' = \tup{M', V, \ldots, \kappa^+, \sqeq, \kappa, f}$. 

Now let $\mc{M}_0' = \tup{M_0', V_0, \ldots, L_0, \sqeq_0, c_0, f_0}$ be a countable elementary substructure of $\mc{M}'$.  Add countably many new constant symbols $d_l$, $l \in L$, and another constant symbol $d^*$.  Let $T$ be the elementary diagram of $\mc{M}_0'$, together with the sentences $\set{d_l \sq d^* : l \in L}$.  Define
\[\Sigma(x) = \set{L(x)} \cup \set{U(f(x))} \cup \set{d(x, d_l) = 1 : l \sq c}.\]
We note that a model of $T$ that omits $\Sigma$ corresponds to a elementary extension of $\mc{M}_0'$ in which $V_0$ is dense in the interpretation of $U$.  The extension is proper because the interpretation of $d^*$ will satisfy $d(f(d^*), f(d_l)) \geq R$ for every $l$, and $f(d^*) \not\in L$, while every $m \in M_0' \setminus L$ satisfies $d(m, d_l) < R$ for some $l$.  We have $V_0$ dense in the interpretation of $U$ because the image of $f$ on elements of $L$ below $c$ is dense in $U$, and omitting $\Sigma$ ensures that no new such elements are added.

\begin{claim}\label{claimApplication}
$\Sigma(x)$ is non-principal over $T$.
\end{claim}
\begin{proof}
We note first that if $t$ is a term that is not a variable symbol or a constant symbol then $T \models \forall x \neg U(t(x))$.  It therefore suffices to show that if $\psi(x)$ is a formula consistent with $T$ and $r \in \mathbb{Q} \cap (0, 1)$, then $T \cup \set{\psi(x) \geq r} \not\models \Sigma(x)$.

Now suppose that $\psi(x)$ is consistent with $T$.  Let us write $\psi(x, d)$ to emphasize that the new constant symbol $d$ may appear.  If either $\psi(x, d) \wedge \neg L(x)$ or $\psi(x, d) \wedge L(x) \wedge \neg U(f(x))$ is consistent with $T$ then by definition of $\Sigma$, $T \cup \set{\psi(x, d)} \not\models \Sigma(x)$ and we are done.  So we may assume that $\psi(x, d) \wedge L(x) \wedge U(f(x))$ is consistent with $T$.  It follows from the definition of $T$ that
\[\mc{M}_0' \models \forall z \in L\,\sup_{y \in L}\,\sup_{x \in L}\,(z \sqeq y \wedge U(f(x)) \wedge \psi(x, y)).\]
By elementary equivalence, $\mc{M}'$ is also a model of this sentence.  Pick $q \in \mathbb{Q} \cap (0, 1)$ such that $q > r$.  For each $\alpha \in \kappa^+$, find $x_q^\alpha \in \kappa^+$ such that
\[\mc{M}' \models \sup_{y \in L}\,(\alpha \sqeq y \wedge U(f(x_q^\alpha))\,\wedge\,\psi(x_q^\alpha, y)) \geq q.\]
This implies that $\mc{M}' \models U(f(x_q^\alpha))$, so by our choice of $f$ we have that $x_q^\alpha < \kappa$.  Since $\kappa^+$ is regular there exists $x_q$ such that for all sufficiently large $\alpha$, $x_q = x_q^\alpha$.  We thus have
\[\mc{M}' \models \forall z \in L\,\sup_{y \in L}\,(z \sqeq y \wedge U(f(x_q)) \wedge \psi(x_q, y)) \geq q.\]
By elementary equivalence,
\[\mc{M}_0' \models \sup_{x \in L}\,\forall {z \in L}\,\sup_{y \in L}\,(z \sqeq y \wedge U(f(x)) \wedge \psi(x, y)) \geq q.\]
Now pick $r' \in \mathbb{Q} \cap (0, 1)$ such that $r < r' < q$.  Then there exists $x_{r'}$ such that
\[\mc{M}_0' \models \forall z \in L \sup_{y \in L} (z \sqeq y \wedge U(f(x_{r'})) \wedge \psi(x_{r'}, y)) \geq r'.\]
This implies that $\mc{M}_0' \models U(f(x_{r'})) = 1$, so there is some $l$ such that $x_{r'} = d_l \sqeq c_0$.  Thus, using that the metric $d$ is discrete in $L_0$,
\[\mc{M}_0' \models \forall z \in L\,\sup_{y \in L}\,(z \sqeq y \wedge \sup_{x \in L}(\psi(x, y) \geq r' \wedge d(x, d_l)=0).\]
We therefore have that $\psi(x, d) \geq r' \wedge d(x, d_l) = 0$ is consistent with $T$.  Since $d(x, d_l) = 1$ appears in $\Sigma$, this shows that $\psi(x, d) \geq r' \not\models \Sigma(x)$, and hence $\psi(x, d)~\geq~r \not\models \Sigma(x)$.
 
\renewcommand{\qedsymbol}{$\dashv$ -- Claim~\ref{claimApplication}}
\end{proof}

By Claim~\ref{claimApplication} and the Omitting Types Theorem (Theorem~\ref{thmOmittingTypes}) there is $\mc{M}_1' \models T$ that omits $\Sigma$.  Repeating the above argument $\omega_1$ times we get an elementary chain $(\mc{M}_\alpha')_{\alpha < \omega_1}$.  For each $\alpha < \omega_1$ let $\mc{M}_\alpha$ denote the reduct of $\mc{M}_\alpha'$ to $S$.  Then $\mc{N} = \bigcup_{\alpha < \omega_1}\mc{M}_\alpha$ is the desired model.
\end{proof}

We note that instead of using a discrete predicate $U$, we could instead have used a two-sorted language, with only notational differences in the proof.  We will use this in our application to non-trivial separable quotients of Banach spaces.  The following result was proved by Ben Yaacov and Iovino \cite{benYaacov2009} in the case of finitary continuous logic.

\begin{cor}\label{corBanach}
Let $X$ and $Y$ be infinite-dimensional Banach spaces with $\op{density}(X) > \op{density}(Y)$.  Let $T : X \to Y$ be a surjective bounded linear operator.  Let $L$ be a countable continuous fragment of $\IL(S)$, where $S$ is a two-sorted signature, each sort of which is the signature of Banach spaces, together with a symbol to represent $T$.  Then there are Banach spaces $X', Y'$ with $Y'$ separable and $X'$ of density $\aleph_1$, and a surjective bounded linear operator $T' : X' \to Y'$, such that $(X, Y, T) \equiv_L (X', Y', T')$.
\end{cor}
\begin{proof}
By Theorem~\ref{thmKeisler} we get normed linear spaces $\widetilde{X}, \widetilde{Y}$ and a bounded linear map $\widetilde{T} : \widetilde{X} \to \widetilde{Y}$ with the desired properties.  Since $L$ is a continuous fragment we may take completions to get the desired spaces $X', Y'$ and operator $T'$.  It remains only to be seen that $T'$ is surjective, but this follows from elementary equivalence in the finitary part of $L$ and the linearity of $T$ (see \cite[Proposition 5.1]{benYaacov2009}).
\end{proof}

We note that if the space $Y$ in the statement of Corollary \ref{corBanach} is already separable then the Downward L\"owenheim-Skolem Theorem suffices to obtain a stronger result:

\begin{cor}
Fix a continuous countable fragment $L$ of $\IL$.  Then every infinite-dimensional separable quotient of a non-separable Banach space $X$ is also a quotient of a Banach space $X'$, where $X'$ has density $\aleph_1$, and $X' \preceq_L X$.
\end{cor}
\begin{proof}
Let $D \subseteq Y$ be countable and dense, and use Downward L\"owenheim-Skolem to find $(X', Y', T') \preceq_L (X, Y, T)$ of density $\aleph_1$ with $D \subseteq Y'$.  By the continuity of the fragment $L$, we may assume that $X'$ and $Y'$ are complete.  It therefore suffices to observe $Y' = Y$.  Indeed, we have $D \subseteq Y' \subseteq Y$, with $D$ dense in $Y$, so $Y'$ is also dense in $Y$.  Since $Y'$ is complete it is closed in $Y$, and hence $Y' = Y$.
\end{proof}
% ----------------------------------------------------------------
\bibliographystyle{amsalpha}
\bibliography{OmittingTypes}
\end{document}